\documentclass[11pt]{article}

\usepackage{amsmath}
\usepackage{amsthm}
\usepackage{amssymb}
\usepackage[left=2.5cm,right=2.5cm,top=3cm,bottom=3cm]{geometry}
\usepackage{amscd}
\usepackage{stmaryrd}
\usepackage[normalem]{ulem}
\usepackage{MnSymbol}
\usepackage{bbm}
\usepackage{color}

\usepackage[utf8]{inputenc}

\usepackage[pdftex]{graphicx}
\usepackage{enumerate}

\usepackage{array}

\usepackage[small,nohug,heads=vee]{diagrams}
\diagramstyle[labelstyle=\scriptstyle]

\theoremstyle{plain} \newtheorem{thm}{Theorem}
\newtheorem{prop}{Proposition} [subsection]
\newtheorem{lem}{Lemma} [subsection]
\newtheorem{corol}{Corollary}

\theoremstyle{remark} 
\theoremstyle{definition} 
\theoremstyle{remark} 
\theoremstyle{remark} \newtheorem*{rem}{Remark}
\theoremstyle{definition} 
\theoremstyle{theorem} \newtheorem*{fact}{Fact}

\newcommand{\comment}[1]{}

\newcommand{\quotient}[2]{{\left.\raisebox{.2em}{$#1$}\middle/\raisebox{-.2em}{$#2$}\right.}}

\newcommand{\abs}[1]{ {\left| #1 \right| } }
\newcommand{\norm}[1] { \abs{\abs{#1}}}

\newcommand{\D}{{\overline{\partial}}}

\newcommand{\Der}[1]{ {\frac{\partial}{\partial #1}} }

\newcommand{\Ps}[1]{{\mathbb P^\ast_{#1}}}

\newcounter{listcounter}

\begin{document}

\author{Benoît Cadorel}
\title{Symmetric differentials on complex hyperbolic manifolds with cusps}
\date{}

\maketitle

\begin{abstract} Let $(X, D)$ be a logarithmic pair, and let $h$ be a smooth metric on $T_{X\setminus D}$. We give a simple criterion on the curvature of $h$ for the bigness of $\Omega_X(\log D)$ or $\Omega_{X}$. As an application, we obtain a metric proof of the bigness of $\Omega_X(\log D)$  on any toroidal compactification of a bounded symmetric domain. Then, we use this singular metric approach to study the bigness and the nefness of $\Omega_X$ in the more specific case of the ball. We obtain effective ramification orders for a cover $X' \longrightarrow X$, étale outside the boundary, to have all its subvarieties with big cotangent bundle. We also prove that $\Omega_{X'}$ is nef if the ramification is high enough. Moreover, the ramification orders we obtain do not depend on the dimension of the ball quotient we consider.
\end{abstract}

\section{Introduction}

For any compact quotient $X$ of a bounded symmetric domain, we know from the work of Brunebarbe, Klingler and Totaro \cite{BKT13}, that the cotangent bundle $\Omega_X$ must be a big vector bundle. The method they use to prove this result consists mainly in computing the curvature of the Bergman metric to show that the bundle must be nef, and then that its higher Segre class must be positive. In the case where $X$ is merely a compactification of a quotient of a bounded symmetric domain, with boundary $D$, the general philosophy of logarithmic pairs says that $\Omega_X(\log D)$ should have positivity properties similar to the compact case. In this spirit, Brunebarbe proves the following in \cite{bru16}:

\begin{thm} [\cite{bru16}] \label{bigthm} Let $(X, D)$ be a toroidal compactification of a quotient of a bounded symmetric domain. Then $\Omega_X (\log D)$ is big.
\end{thm}

Brunebarbe's proof relies on a close study of some well suited variations of Hodge structure. One purpose of this paper is to give a metric approach to this result,  generalizing the one of \cite{BKT13}. Actually, a theorem of Boucksom \cite{bou02} indicates that we can estimate the volume of a given pseudo-effective line bundle, by the maximal power of the curvature of a suitable singular metric, integrated outside its singularities. Applying these ideas, we can prove the following simple criterion for the bigness of the cotangent bundle of a logarithmic pair.

\begin{thm} \label{thmsingmetric} Let $(X, D)$ be a logarithmic pair. Assume that $T_X |_{X \setminus D}$ admits a smooth Kähler metric $h$ satisfying the following hypotheses:
\begin{enumerate}
\item $h$ has negative holomorphic sectional curvature on $X \setminus D$, bounded by a constant $-A$ ;
\item $h$ has non-positive bisectional curvature;
\setcounter{listcounter}{\theenumi}
\end{enumerate}
Then $\Omega_{X} (\log D)$ is a big vector bundle. In addition, if
\begin{enumerate}
\setcounter{enumi}{\thelistcounter}
\item $h$, seen as a metric on $T_{X}$, is locally bounded;
\end{enumerate}
then $\Omega_{X}$ is big.

\end{thm}

Remark that this result, coupled with a theorem of Campana and P\u{a}un \cite{campau15}, implies that a logarithmic pair $(X, D)$ with a Kähler metric satisfying the first two hypotheses of Theorem \ref{thmsingmetric}, must have $K_X + D$ big. This can be seen as a weak logarithmic version of a recent theorem of Wu and Yau \cite{wy15}, stating that a projective manifold admitting a Kähler metric with negative holomorphic sectional curvature must have an ample canonical bundle.

In the case of a quotient of a bounded symmetric domain, the Bergman metric on the open part of a compactification satisfies all the properties we need to apply Theorem \ref{thmsingmetric}. Thus, it seems that the use of singular metrics is well suited to study the positivity properties of the toroidal compactifications of bounded symmetric domains. In particular, we will see that for toroidal compactification of a quotient of the \emph{ball}, we can obtain effective results for the general notions of positivity of the cotangent bundle. \\

If a quotient $X = \quotient{\mathbb B^n}{\Gamma}$ is compact, it is well known that the Bergman metric on $\mathbb B^n$ will induce negativity properties on $T_X$. In particular, the bundles $K_{X}$, $\Omega_{X}$ will be ample, $X$ will be Kobayashi hyperbolic, and so on. If the group $\Gamma$ is not co-compact, it is legitimate to ask to what extent these properties are preserved under the toroidal compactification. More precisely, given such a toroidal compactification $\overline{X} = \overline{\quotient{\mathbb B^n}{\Gamma}}$, we would like to study the general notions of positivity for the classical bundles supported by $\overline{X}$.

In the simple case of curves, we know that $K_{\overline{X}}$ has \emph{a priori} no reason to be even nef (i.e. to have non-negative degree): it suffices to consider $\overline{X} = \mathbb P^1$, and $X = \mathbb P^{1} \setminus \left\{0, 1 \infty \right\}$, which is a quotient of the unit disk. In the case of surfaces, Hirzebruch considers in \cite{hir84} the blowing-up of a product of two elliptic curves at a point. By using logarithmic Yau-Miyaoka's inequality, he shows that such a manifold is a toroidal compactification of a quotient of $\mathbb B^2$. This provides an example of a toroidal compactification of a ball quotient for which $K_{\overline{X}}$ is neither big nor nef. However, this particular feature of $K_{\overline{X}}$ is specific to small dimensions: Di Cerbo and  Di Cerbo prove in \cite{dcdc15b} that $K_{\overline{X}}$ must always be nef for $n \geq 3$. Using their work, Bakker and Tsimerman show in turn (\cite{baktsi15}) that $K_{\overline{X}}$ is big for $n \geq 4$, and even ample if $n \geq 6$.

We propose to study the various notions of positivity for the cotangent bundle $\Omega_{\overline{X}}$, on a given toroidal compactification $\overline{X}$ of a ball quotient by a lattice with unipotent parabolic elements. First of all, the results of  \cite{bou02} will permit us to estimate the intersection numbers of the logarithmic tautological bundle with curves $C \subset \mathbb P \left( T_{\overline{X}}( - \log D)\right)$.  The nefness of the logarithmic cotangent bundle of $\overline{X}$ will follow naturally.

\begin{thm} \label{nefcotangentlog} Let $\Gamma \subset \mathrm{Aut} (\mathbb B^n)$ be a lattice with only unipotent parabolic elements. Then, if $(\overline{X}, D)$ is the toroidal compactification of $X = \quotient{\mathbb B^n}{\Gamma}$, the logarithmic cotangent bundle $\Omega_{\overline{X}} (\log D)$ is nef. 
\end{thm}

Using singular metrics related to the Bergman metric on $\mathbb B^n$ permits us to compare the curvature of $\Omega_{X}$ and $K_X$ on the open part $X \subset \overline{X}$. The results of \cite{baktsi15} provide us with effective estimate on the positivity of $K_{\overline{X}}$, and of its linear combinations with $D$, that we can transpose to $\Omega_{\overline{X}}$. Since the cotangent bundle behaves well under restriction to subvarieties, we can prove the following statement. 

\begin{thm} \label{thmsubvar} Let $X'$ be a quotient of $\mathbb B^n$ by a lattice with unipotent parabolic elements, and let $X \longrightarrow X'$ be an étale cover ramifying at order at least $l$ on any boundary component. Assume that
\begin{enumerate}[(i)]
\item $l \geq 6$ if $n = 4$,
\item $ l \geq 7$ if $n \geq 5$, or $n \in \left\{2, 3\right\}$.
\end{enumerate}
Then, for any subvariety $V$ of $\overline{X}$, not included in $D$, any resolution of $V$ has big cotangent bundle. In particular, by \cite{campau15}, any such subvariety is of general type.
\end{thm}

This result is an effective version, in the case of the ball, of a recent work of Brunebarbe, who proves in \cite{bru16a} that if $\Omega$ is a bounded symmetric domain, and $\Gamma \subset \mathrm{Aut} (\Omega)$ is a neat lattice, then for all $\Gamma' \subset \Gamma$ of sufficiently high finite index, all subvarieties of a toroidal compactification $\overline{ \quotient{\Omega}{\Gamma}}$ are of general type if they are not included in the boundary.

Now that we know that the logarithmic cotangent bundle is nef on a compactification of a quotient of $\mathbb B^n$, we want to prove similar claims on the standard cotangent bundle. One natural way to complete our study, is to resolve the birational transformation $\mathbb P( T_{\overline{X}} (- \log D)) \dashrightarrow \mathbb P( T_{\overline{X}})$, and to use this resolution to relate the two tautological bundles on these projectivized spaces. This resolution will be introduced in Section \ref{biratsection}. This will provide us with useful identities of intersection numbers, which will give us a bound on the ramification needed for $\Omega_{\overline X}$ to be nef.

\begin{thm} \label{thmnef} Let $X \overset{\sigma}{\longrightarrow} X'$ be a finite cover of a ball quotient, ramifying to an order larger as in $7$ on the boundary $D' \subset \overline{X'}$. Then $\Omega_{\overline X}$ is nef.
\end{thm}

Finally, Theorem \ref{thmnef} allows us to refine Theorem \ref{thmsubvar}, if we restrict our study to \emph{immersed} subvarieties of $\overline{X}$. Recall that a vector bundle $E$ is said to be \emph{ample modulo an analytic subset} $Z$, if some power of $\mathcal O(1)$ on $\mathbb P(E^\ast)$ induces a rational map which is an embedding off of $Z$.

\begin{corol} \label{corolamplemod} Under the same hypotheses than Theorem \ref{thmnef}, any immersed subvariety $V \overset{f}{\longrightarrow} \overline{X}$, not included in the boundary, is such that $\Omega_V$ is ample modulo the boundary $f^{-1} (D)$.
\end{corol}

We see that Theorems \ref{thmsubvar} and \ref{thmnef} can be related to a result of \cite{baktsi15} about the Green-Griffiths conjecture on the pairs $(\overline{X}, D)$. In their article, Bakker and Tsimerman actually use a theorem of Nadel \cite{nad89} to prove that if $\dim X = 3$ (resp. $\dim X = 4, 5$, resp. $\dim X \geq 6$), $\overline{X}$ will verify the Green-Griffiths conjecture when the ramification order $l$ satisfies $l \geq 2$ (resp. $l \geq 3$, resp. $l \geq 4$). In particular, it implies that, with the same ramification orders, all curves not included in the boundary are hyperbolic. The bounds of \cite{baktsi15} are consequently smaller than ours in the case of curves, but our method has the advantage of working for submanifolds of any dimension. \\

\begin{bfseries} Acknowledgments.\end{bfseries} The author would like to thank his advisor Erwan Rousseau for his guidance and his fruitful ideas, and Julien Grivaux for his support and enlightening discussions on many aspects of this work.

\section{Compactifications of ball quotients} \label{sectioncompact}
\subsection{Construction of the toroidal compactification}

We recall some results on the structure of the toroidal compactification of a quotient of the complex unit ball. Let $\Gamma \subset \mathrm{PU}(n,1)$ be a group of automorphisms of the ball, with finite covolume. As explained in \cite{mok12} and \cite{dcdc15}, if we assume that all parabolic isometries of $\Gamma$ are unipotent, it is possible to compactify the quotient $\quotient{\mathbb B^n}{\Gamma}$ using a construction similar to the one of \cite{mumford}, which we can find in full detail in \cite{mok12}. If $\Gamma$ is supposed to be a neat arithmetic subgroup of $\mathrm{Aut}(\mathbb B^n)$, this assumption will always be verified. \\

From now on, we will assume that $\Gamma$ is a lattice of automorphisms of $\mathbb B^n$ with unipotent parabolic isometries. Let $X = \quotient{\mathbb B^n}{\Gamma}.$ The toroidal compactification of $X$ consists in adding to it a finite number of abelian varieties at its cusps, to obtain a smooth manifold $\overline{X}$. Let us  describe the structure of $\overline{X}$ in the neighborhood of such a cusp.

For any $N > 0$, let
\begin{equation} \label{Siegelrep}
S^{(N)} = \left\{ (z', z_n) \in \mathbb C^{n-1} \times \mathbb C ; l(z', z_n) > N \right\},
\end{equation}
with $l(z', z_n) = \mathrm{Im} z_n - \norm{z'}^{2}.$ The open set $S^{(0)}$ is a Siegel domain representation of $\mathbb B^n$ with respect to a given base point $b \in \partial \mathbb B^n$, and the family $(S^{(N)})_N$ represents the family of horoballs of $\mathbb B^n$ at the point $b$. 

There exists a finite number of conjugacy classes of maximal parabolic subgroups $\Gamma_i \subset \Gamma$, each one of them corresponding to a cusp $C_i$ of $X$. Let $\Gamma_b \subset \Gamma$ be such a group, fixing some $b \in \partial \mathbb B^n$. Then, for a certain $N > 0$, $\Gamma_b$ fixes the horoball $S^{(N)}$, where the Siegel representation \eqref{Siegelrep} is taken so that $0 \in \mathbb C^{n-1} \times \mathbb C$ corresponds to $b$.

The stabilizer of $b$ in $\Gamma$ acts on $S^{(N)}$ as the semi-direct product of two group actions, which we will now describe. The first one of these is an action of $\mathbb Z$, defined by
$$
k \cdot (z', z_n) = (z', z_n + k \tau),
$$
where $\tau \in \mathbb R^{\ast}_+$ is some parameter depending on $b$. Let $G^{(N)} = \quotient{S^{(N)}}{\mathbb Z}$, with its natural analytic structure.

We have $G^{(N)} \cong \left\{ (w', w_n) \in \mathbb C^{n-1} \times \mathbb C^\ast ; \norm{ (w', w_n) }_{\mu} < e^{- \frac{2\pi}{\tau} N} \right\}$, where $\norm{ (w', w_n ) }_\mu = \abs{w_n} e^{\frac{2\pi}{\tau} \norm{w'}^2}$. The projection is realized by the following holomorphic application:
$$
\begin{array}{ccc}
S^{(N)} & \overset{\Psi}{\longrightarrow} & G^{(N)} \\
(w', w_n) & \longmapsto & \left( z', e^{\frac{2 i \pi z_n}{\tau}} \right).
\end{array}
$$
Let $\widehat{G^{(N)}} = \left\{ (w', w_n) \in \mathbb C^{n-1} \times \mathbb C ; \norm{ (w', w_n) }_{\mu} < e^{- \frac{2\pi}{\tau} N} \right\}$. Thus, if we note $D_0 = \left\{ w_n = 0 \right\} \subset \widehat{G^{(N)}}$, we see easily that the differential of $\Psi$ send surjectively $T_{S^{(N)}}$ onto $T_{G^{(N)}} (- \log D_0)$.

The second group action comes from a lattice $\Lambda_b \subset \mathbb C^{n-1}$, and can be written
$$
a \cdot (z', z_n) = \left(z' + a, z_n + i \norm{a}^2 + 2 i \overline{a} \cdot z' \right),
$$
for $a \in \Lambda_b$, $(z', z_n) \in S^{(N)}$.
The stabilizer of $b$ in $\Gamma$ acts on $S^{(N)}$ as the semi-direct product of these two previous actions. Consequently, the action of $\Lambda_b$ goes to the quotient $\quotient{S^{(N)}}{\mathbb Z} \cong G^{(N)}$, and we can write its action on $G^{(N)}$ as
\begin{equation} \label{actionLambda}
a \cdot (w', w_n) = \left(w' + a, e^{- \frac{2 \pi}{\tau} \norm{a}^2} e^{- 4 \pi \frac{\overline{a} \cdot z'}{\tau}} w_n \right).
\end{equation}

The action of $\Lambda_b$ on $G^{(N)}$ extends naturally to an action on $\widehat{G^{(N)}}$. We can define the open manifold $\Omega^{(N)}_b$ to be the quotient $\quotient{\widehat{G^{(N)}}}{\Lambda_b}$.

The subspace $D_0 \subset \widehat{G^{(N)}}$ goes to the quotient by $\Lambda_b$, to give an abelian variety $D_b \; = \; \quotient{D_0}{\Lambda_b} \hookrightarrow \Omega^{(N)}_b$. Moreover, the embedding of the horoball $S^{(N)} \hookrightarrow \mathbb B^n$ induces an embedding of the quotient
$$
\Omega^{(N)} \setminus D_b \;  = \; \quotient{G^{(N)}}{\Lambda_b} \; \hookrightarrow \; X.
$$

The \emph{toroidal compactification} of $X$ is defined to be the glueing of the manifolds $\Omega^{(N)}_{b_i}$ on $X$ along the open subsets $\Omega^{(N)}_{b_i} \setminus T_{b_i}$, where the $b_i \in \partial \mathbb B^n$ span a family of representatives of the cusps. Let us denote by $\overline{X}$ this compactification. We see that, as sets, we have
$$
\overline{X} = X \sqcup \bigsqcup_i D_{b_i}.
$$

Let us denote by $D = \bigsqcup_i D_{b_i}$ the compactifying divisor of $\overline{X}$. This divisor is  a disjoint union of abelian varieties. \\

\begin{bfseries}Terminology. \end{bfseries} \begin{enumerate}
\item In the rest of this paper, a \emph{ball quotient} will always mean a quotient of $\mathbb B^n$ by a subgroup of $\mathrm{PU}(n, 1)$ with finite covolume and unipotent parabolic isometries. 
\item Unless otherwise specified (e.g. in Section \ref{bigsection}), a \emph{toroidal compactification} will always be a toroidal compactification of a ball quotient, as defined in this section.
\end{enumerate}

\subsection{Local coordinates. Bergman metric} \label{localcoordsect}

Let $D_b$ be a component of $D$, and let $w_0 \in D_b$ be any point of this component. In some neighborhood $U$ of $x_0$, we can consider local coordinates $(w', w_n)$, coming from the global coordinates on $\widehat{ G^{(N)}_b }$. We will describe explicitly the action of $\Lambda_b$ on the logarithmic tangent bundle of $U$ in these coordinates.

First, we study the action of this group on $T_{G^{(N)}} (- \log D_0)$. By  \eqref{actionLambda}, it can be expressed as
$$
\left\{ \begin{array}{ccc}
 a \cdot \left. \frac{\partial}{\partial w'_i} \right|_{x} & = &  \sum_j {\left. \frac{\partial {w'}^\sharp_j}{\partial w'_i} \frac{\partial}{\partial {w'}^\sharp_j} \right|_{a \cdot x} }+  \left. \frac{\partial {w'}^\sharp_j}{\partial w'_i} \frac{\partial}{\partial {w'}^\sharp_j} \right|_{a \cdot x} = \frac{\partial}{\partial {w'}^\sharp_i} - \frac{4 \pi \bar{a_i}}{\tau} w^{\sharp}_n \left( \frac{\partial}{\partial {w}^\sharp_n} \right)_{a \cdot x} \\
 a \cdot \left( w_n \frac{\partial}{\partial w_n} \right)_x &  =  & \left( w^{\sharp}_n \frac{\partial}{\partial w^{\sharp}_n} \right)_{a \cdot x} \hfill,
 \end{array} \right.
$$
where $(w^\sharp_i)$ is the family of coordinates at the point $a \cdot x$. After taking the quotient by $\Lambda_b$, we see that
$$
(e_j)_{1 \leq j \leq n} = \left( \left( \Der{w_j} - \overline{w_j} \left( \frac{4 \pi}{\tau} w_n \Der{w_n}\right) \right)_{1 \leq j \leq n-1}, \frac{4 \pi}{\tau} w_n \Der{w_n} \right)
$$
is well defined on the whole $\Omega^{(N)}_b$, and realizes a \emph{smooth} frame for $T_{\overline{X}} (- \log D)$ on $\Omega^{(N)}_b$ for some $N > 0$ large enough.

Recall that on the ball $\mathbb B^n$, with standard coordinates $(z_j)$, the Bergman metric is given by, up to a normalization choice:
\begin{equation} \label{bergmandef}
h_{\mathrm{Berg}} = \frac{ (1 - \norm{z}^2) \sum_j dz_j \otimes d\overline{z}_j + \left( \sum_j \overline{z_j} d z_j \right) \otimes \left( \sum_j z_j d \overline{z_j} \right) }{ (1 - \norm{z}^2)^2}.
\end{equation}
With this particular choice of normalization, the metric has constant holomorphic sectional curvature equal to $-4$, and we also have $\mathrm{Ric}(h_{\mathrm{Berg}}) = - 2 (n + 1) \omega_{\mathrm{Berg}}$, where $\omega_{\mathrm{Berg}}$ is the Kähler form associated with the Bergman metric. \\

The smooth frame $(e_j)_j$ permits to express the Bergman metric on $\Omega^{(N)} \setminus D_b$. Indeed, as we can see from \cite{mok12}, we have the following proposition:

\begin{prop} \label{smoothframe} The Bergman metric on $\mathbb B^n$ induces a singular hermitian metric on $T_{\overline{X}} (- \log D)$, whose expression in the frame $(e_j)_j$ admits the diagonal form
\begin{equation} \label{diagonalformeq}
\left(H_{ij}\right) \; = \;  \left( h_{Berg} (e_i, e_j) \right) \; = \;  \mathrm{diag} (l(w)^{-1}, ..., l(w)^{-1},l(w)^{-2}),
\end{equation}
with, for any $w = (w', w_n) \in \Omega^{(N)}_b \setminus D_b$, $l(w) = - \frac{\tau}{4 \pi} \log \norm{w}_\mu^2$. 
\end{prop}

\begin{rem}
Even though the metric $\norm{\cdot}_\mu$ is \emph{a priori} defined only on $S^{(N)}$, it is invariant under the actions of $\mathbb Z$ and $\Lambda_b$, so it is legitimate to express the norm $\norm{w}_\mu$ for any $w \in \Omega^{(N)}_b \setminus D_b$. \\
\end{rem}

Later on, we will need to compute the intersection numbers of $K_{\overline{X}} + D$ in terms of the Bergman metric on $X \subset \overline{X}$. The following proposition, which comes from Mumford's work \cite{mum77}, will be useful for this purpose.

\begin{prop} \label{propintbergman} Let $(\overline{X}, D)$ be a toroidal compactification, and let $\overline{V} \overset{f}{\longrightarrow} \overline{X}$ be a generically injective holomorphic map, from a complex manifold of dimension $p$, such that $f( \overline{V} ) \not\subset D$. Let $V = f^{-1} (D)$. Then we have
$$
\left( K_{\overline{X}} + D \right)^{p} \cdot \left[ f(\overline{V}) \right] = \int_{V} \left(\frac{i}{2 \pi} f^{\ast} \Theta(\det h_{Berg}^\ast) \right)^{p} = \frac{(n+1)^p}{\pi^p} \int_{V} f^\ast \omega_{Berg}^{p}.
$$
\end{prop}

The first equality actually comes from the fact that $h^\ast_{Berg}$ is a \emph{good} metric on $\Omega_{\overline{X}} (\log D)$ in the sense of \cite{mum77}. The second equality follows because $\mathrm{Ric} (h_{Berg}) = -2(n+1) \omega_{Berg}$.

\section{Bigness of the cotangent bundles} \label{bigsection}

In this section, we use singular metrics to study the bigness of the standard and logarithmic cotangent bundle of a logarithmic pair $(X, D)$. We will see that general assumptions on the negativity of the curvature of $X \setminus D$, are already sufficient to prove that $\Omega_X(\log D)$ is big. \\

\begin{bfseries} Terminology.
\end{bfseries}
We call a \emph{log-pair} the data of a pair $(X, D)$, where $X$ is a smooth complex projective manifold, and $D \subset X$ a divisor with simple normal crossings. If $D$ is smooth, we say that the log-pair $(X,D)$ has \emph{smooth boundary}. 

\subsection{Singular metrics on the tangent bundles}

The following result relates the bigness of the standard and logarithmic cotangent bundles of a given log-pair $(X, D)$, to the negativity of the curvature of a given Kähler metric on the open part $X \setminus D$. This result is a generalization of a theorem of \cite{BKT13}: we will use a criterion for bigness of \cite{bou02}, coupled with the well known Ahlfors-Schwarz lemma, to extend the field of application of their proof. This will give a proof of the following theorem, which is a slightly more general version of Theorem \ref{thmsingmetric}.

\begin{thm} \label{thmsingmetricnonK} Let $(X, D)$ be a logarithmic pair. Assume that $T_X |_{X \setminus D}$ admits a smooth metric $h$ (not necessarily Kähler) satisfying the following hypotheses:
\begin{enumerate}
\item $h$ has negative holomorphic sectional curvature $H$ on $X \setminus D$, bounded by a constant $-A$ ;
\item $h$ has non-positive bisectional curvature $B$ ;
\item $h$ has negative bisectional curvature at some point of $\mathbb P(T_X |_{X \setminus D})$ i.e. there exist $x_0 \in X \setminus D$, $v_0 \in T_{x_0} X \setminus \left\{0\right\}$ such that
$$
\forall w \in T_{x_0} X \setminus \left\{ 0 \right\}, \; B (v_0, w) < 0.
$$
\setcounter{listcounter}{\theenumi}
\end{enumerate}
Then $\Omega_{X} (\log D)$ is big. In addition, if
\begin{enumerate}
\setcounter{enumi}{\thelistcounter}
\item $h$, seen as a metric on $T_{X}$, is locally bounded;
\end{enumerate}
then $\Omega_{X}$ is big.
\end{thm}

\begin{rem}
By \cite{BKT13}, if the metric $h$ is supposed to be \emph{Kähler}, the first two hypotheses of Theorem \ref{thmsingmetricnonK} actually imply the third one. Thus, Theorem \ref{thmsingmetric} is a consequence of Theorem \ref{thmsingmetricnonK}.
\end{rem}

Before proving Theorem \ref{thmsingmetricnonK}, let us begin by recalling some well known growth properties of metrics with negative holomorphic sectional curvature, derived from the Ahlfors-Schwarz lemma.

\begin{prop} [Ahlfors-Schwarz lemma] Let $\mathbb H$ be a model of the Poincaré half-plane, with its canonical metric $\omega_P$. Let $h$ be another smooth metric on $T_{\mathbb H}$, with negative sectional curvature bounded by a constant $-A$. Then, there exists a constant $C > 0$, depending only on $A$, such that
$$
h \leq C \omega_P.
$$ 

In particular, if $\Delta$ is the unit disk of $\mathbb C$, and if $h$ is a metric on $T_{\Delta}$ with bounded negative curvature as above, there exists $C > 0$ such that
\begin{equation} \label{boundpoincare}
h(z) \leq \frac{C}{(1 - \abs{z}^2)^2}.
\end{equation}
Similarly, if $\Delta^\ast$ is the punctured unit disk, any such metric on $T_{\Delta^\ast}$ is bounded as
$$
h(z) \leq \frac{C}{ \abs{z}^2 \abs{\log \abs{z}}^2}.
$$

\end{prop}

Now, let $\Delta^n$ be the unit polydisk in $\mathbb C^n$, with the coordinates $(z_1, ..., z_n)$, and let $U = (\Delta^\ast)^m \times \Delta^{n -m}$ be the complement of $D = \left\{ z_1 ... z_m = 0 \right\}$. We introduce the Poincaré metric $h^{(p)}$ on $U$, defined by its Kähler form
$$
\omega^{(p)} = \sum^{m}_{k = 1} \frac{ \frac{i}{2} d z_k \wedge d \overline{z}_k}{ \abs{z_k}^2 \abs{\log{\abs{z_k}}}^2} + \sum_{k = m + 1}^n \frac{i}{2} d z_k \wedge d \overline{z_k}.
$$

\begin{prop} Let $h$ be a smooth metric on $T_U$, with holomorphic sectional curvature bounded from above by a negative constant $-A$. Then $h$ has Poincaré growth, i.e. for any $x \in D$, there exists a constant $C$ (depending only on $A$) such that for any vector fields $\xi$ and $\eta$ on $U$, we have
\begin{equation} \label{growthmetric}
\abs{ h(\xi, \eta) }^2 \leq C \omega^{(p)} (\xi, \xi ) \omega^{(p)}( \eta, \eta).
\end{equation}
in the neighborhood of $x$.
\end{prop}
\begin{proof}
Applying Cauchy-Schwarz, we see that it suffices to prove that for any vector field $\eta$, we locally have $\norm{ \eta}_h \leq C \norm{ \eta}_{(p)}$. Moreover, we can clearly suppose $\eta$ constant. Let $\eta = \sum_{j} a_j \Der{z_j}$ be such a constant vector field. Then
$$
\norm{\eta}^2_h \leq n^2 \sum_{j} \norm{ a_j \Der{z_j} }^2_h.
$$
Thus, it suffices to prove the result for $\eta = \Der{z_j}$ for any $j \in \left[|1, n \right|]$. Let $x_0 \in U$, and let $U$ be a neighborhood of $x_0$ on which $\norm{x}_{\infty}$ is bounded by a constant $B$, for any $x \in U$.

If $j \in \left[| 1, m \right|]$, we apply the Ahlfors-Schwarz lemma to the punctured disk passing through $x$ and directed by $\Der{z_j}$ to obtain 
$$
\norm{\Der{z_j}}_h (x) \leq C \frac{1}{ \abs{z_j}^2 \abs{ \log \abs{ z_j}}^2},
$$
on $U$, for some $C$ depending only on $A$. Similarly, if $j \in \left[| m + 1, n \right|]$ we see from \eqref{boundpoincare} that $\norm{\Der{z_j}}_h$ must be bounded from above by
$$
\norm{\Der{z_j}}_h(x) \leq C \frac{1}{(1 - \abs{z}^2)^2} \leq \frac{C}{(1 - B^2)^2},
$$
with $C$ depending only on $A$. This proves the result. 
\end{proof}

\begin{corol} \label{corolbounded} Let $\Delta^n$ and $D \subset \Delta^n$ be as above, and let $h$ be a smooth metric on $T_{\Delta^n \setminus D}$, which we suppose to have negative sectional curvature bounded by $-A$. Then for any vector field $\xi$ of $T_X (-\log D)$, $\norm{\xi}_h$ is bounded in the neighborhood of any point of $D$.
\end{corol}
\begin{proof}
It suffices to apply \eqref{growthmetric} on the vectors of the canonical frame $\left( \left(z_j \Der{z_j} \right)_{1 \leq j \leq m}, \left( \Der{z_j} \right)_{m \leq j \leq n} \right)$, and to remark that $\omega^{(p)}$ is bounded on these vectors.
\end{proof}

We now prove that under the first three assumptions of Theorem \ref{thmsingmetricnonK}, $\Omega_{X} (\log D)$ is big. Let $Y = \mathbb P( T_X( - \log D) ) \overset{p}{\longrightarrow} X$ and let $\mathcal O(1)$ be the tautological bundle of this projectivized space.
\begin{lem} The line bundle $\mathcal O(1)$ is pseudo-effective on $Y$. 
\end{lem}

\begin{proof}
Let $\widehat{h}$ be the metric induced by $h$ on the tautological bundle $\mathcal O(-1) \longrightarrow Y$. Remark that $\widehat{h}$ is not defined on $p^{-1} (D)$. Denote by $\widehat{h}^\ast$ the dual of this metric ; it is determined locally by the norm of a non-vanishing section of $\mathcal O(1)$. More specifically, if $(x, \left[ v \right]) \in p^{-1} (D)$, choose a section $\sigma$ of $T_X (- \log D)$, non vanishing around $x$, such that $\sigma(x) = v$. Then $\sigma$ induces a local section $\widehat{\sigma}$ of $\mathcal O(-1)$ around $(x, \left[ v \right])$, whose dual section we will denote by $\widehat{\sigma}$. Locally, the norm of $\widehat{\sigma}^\ast$ is given by
$$
\norm{ \sigma^\ast}_{\widehat{h}^\ast} = \frac{1}{\norm{ \sigma }_{h} },
$$
where $\widehat{\sigma}^\ast$ is the section of $\mathcal O(1)$ dual to $\widehat{\sigma}$. Then, on $p^{-1} (X \setminus D)$, the curvature of $(\mathcal O(1), \widehat{h}^\ast)$ is determined near $(x, [v])$ by
$$
\frac{i}{2} \Theta( \widehat{h}^\ast) \; \overset{loc}{=} \; \frac{i}{2} \D \partial \log \norm{ \widehat{\sigma}^\ast }_{\widehat{h}^\ast} = \frac{i}{2} \partial \D \log \norm{\sigma}_{h}.
$$  
We can develop this expression, to obtain
\begin{equation} \label{curvatureprojectivized}
\frac{i}{2} \Theta( \widehat{h}^\ast) \cdot (\xi, \xi) \overset{loc}{=} - \frac{i}{2} \frac{\left< \sigma, \Theta(h) \cdot (p_\ast \xi, p_\ast \xi) \sigma \right>_h}{\norm{\sigma}^2_h} + \omega^{FS}_h (\xi^{vert}, \xi^{vert}).
\end{equation}
The first term appearing in the right hand side of this equation is equal to $B(\sigma, p_\ast \xi) \norm{p_\ast \xi}_h$, where $B$ is the bisectional curvature of $h$. It is non-negative by our hypothesis. The second term, related to the Fubini-Study metric on the fibers, is also non-negative. This implies that $i \D \partial \log \norm{ \widehat{ \sigma}^\ast}_{\widehat{h}^\ast}^2 \geq 0$, i.e. that $- \log \norm{ \widehat{\sigma}^\ast }_{\widehat{h}^\ast}^2$ is plurisubharmonic on $Y \setminus p^{-1} (D)$. Moreover, $\norm{ \widehat{\sigma} ^\ast}_{\widehat{h}^\ast}^2 = \frac{1}{\norm{\sigma}_{h}^2}$ is locally bounded from below by Corollary \ref{corolbounded}, so $- \log \norm{ \sigma^\ast}_{\widehat{h}^\ast}^2$ is bounded from above. By the usual properties of bounded plurisubharmonic functions, we see that this last function extends uniquely on $p^{-1} (D)$  to a plurisubharmonic function, defined locally on $Y$.

Consequently,  we can write $\widehat{h}^\ast \overset{loc}{=} e^{- \Psi}$, with $\Psi$ plurisubharmonic. This implies in particular that $\widehat{h}^\ast$ is a \emph{singular} metric on $\mathcal O(1)$, with positive curvature in the sense of currents. By \cite{dem92}, this implies in turn that $\mathcal O(1)$ is a pseudo-effective line bundle.
\end{proof}

To conclude, we use the following theorem of Boucksom \cite{bou02}:
\begin{thm} [\cite{bou02}] \label{boucksomthm} Let $L$ be a pseudo-effective line bundle on a compact Kähler manifold $M$ of dimension $n$. Then, for any closed positive current $T \in c_1 (L)$, if we denote by $T_{ac}$ the absolutely continuous part of $T$, the powers $T_{ac}^k$ have bounded mass on $M$.
 
Moreover, the volume of $L$ is equal to
$$
\mathrm{vol}(L) = \max_{T} \int_{M} T^n_{ac},
$$
where $T$ ranges among the positive closed $(1, 1)$-currents representing $c_1(L)$.
\end{thm}

\begin{proof} [Proof of Theorem \ref{thmsingmetricnonK}] Let $T = \frac{i}{2\pi} \Theta_c( \widehat{h}^\ast )$, where by $\Theta_c$ we mean the curvature in the sense of currents. Since $p^{-1}(D)$ has zero Lebesgue measure, for any $k$, $T_{ac}^k$ is the current of integration against $\left[ \frac{i}{2\pi}\Theta( \widehat{h}^\ast) \right] ^k$ on $Y \setminus p^{-1}(D)$. In particular,
$$
\int_{Y} T^{2n-1}_{ac} = \int_{Y \setminus p^{-1} (D)} \left( \frac{i}{2\pi} \Theta( \widehat{h}^\ast) \right)^{2n - 1}.
$$
Remark that Theorem \ref{boucksomthm} implies that this last integral converges. By \eqref{curvatureprojectivized}, we have
$$
\frac{i}{2} \Theta ( \widehat{h}^\ast )_{(x, [v])} (\xi, \xi) = - \norm{ p_{\ast} \xi}^2 B(v, p_{\ast} \xi) + \omega^{FS}_h (\xi^{vert}, \xi^{vert})  
$$
and since $h$ has non-positive bisectional curvature, the $(2n-1, 2n-1)$-form $\left( \frac{i}{2} \Theta( \widehat{h}^\ast) \right)^{2n - 1}$ is non-negative on $Y \setminus p^{-1} (D)$. Moreover, by our third hypothesis, this form is positive at the point $(x_0, \left[v_0\right]) \in Y \setminus p^{-1} (D)$. 

This means, because of Theorem \ref{boucksomthm}, that
$$
\mathrm{vol}\left( \mathcal O(1) \right) \geq \int_{Y \setminus p^{-1} (D) } \left( \frac{i}{2\pi} \Theta( \widehat{h}^\ast) \right)^{2n -1} > 0.
$$
Thus, $\mathcal O(1)$ has positive volume, hence is big on $Y$. This proves the first assertion of Theorem \ref{thmsingmetric}.

Now, assume that $h$, seen as a metric on $T_X$, is locally bounded near $D$. As before, it follows from our second hypothesis that $h$ induces a metric $\widehat{h}^\ast_0$ on the tautological bundle $\mathcal O(1) \longrightarrow \mathbb P(T_X)$, with positive curvature above $X \setminus D$. If $p_0: \mathbb P(T_X) \longrightarrow X$ is the canonical projection, we see that $\widehat{h}^\ast_0$ can locally be written
$$
\widehat{h}^\ast_0  \overset{loc}{=} e^{- \Psi_0},
$$
with $\Psi_0$ plurisubharmonic on $p^{-1} (X \setminus D)$. Because of our fourth hypothesis, we see that $\Psi_0$ must be bounded from above near any point of $p^{-1} (D)$, and thus, as before, it must extend into a plurisubharmonic function near any such point. This implies that the tautological bundle $\mathcal O(1)$ is pseudo-effective. Applying Theorem \ref{boucksomthm}, we obtain that this line bundle has positive volume. This ends the proof.
\end{proof}

We can now give our metric proof of Theorem \ref{bigthm}. If $\Omega$ is a bounded symmetric domain, its Bergman metric $h_{\Omega}$ is a Kähler metric satisfying the first two hypotheses of Theorem \ref{thmsingmetric}. Therefore, for any quotient $X$ of $\Omega$ by a subgroup $\Gamma \subset \mathrm{Aut} (\Omega)$, the metric $h_{X}$ induced on $X$ by $h_{\Omega}$ satisfies those same hypotheses. If $\overline{X} = X \sqcup D$ is any smooth compactification of $X$, with $D$ a divisor with simple normal crossings, Theorem \ref{thmsingmetric} implies that $\Omega_{\overline{X}}(\log D)$ is big. This proves Theorem \ref{bigthm}. \\

We finish this section by  a result which will be central in our study of the nefness of the cotangent bundles of a toroidal compactification.

\begin{prop} \label{integralsubmanifold} Let $(X, D)$ be a pair satisfying the hypotheses 1 and 2 of Theorem  \ref{thmsingmetric}. Let $Y = \mathbb P(T_X (- \log D))$, with its canonical projection $p$ onto $X$. Let $f : V \longrightarrow Y$ a generically finite morphism from a smooth complex manifold onto a subvariety $f(V) \subset Y$, not included in $p^{-1} (D)$. Then $f^\ast \widehat{h}^\ast$ induces a singular metric on $\mathcal O(1)$, and
$$
\mathrm{vol} \left( f^\ast \mathcal O(1) \right) \geq \int_{f^{-1}  \left(Y \setminus p^{-1} (D) \right) \cap V_S  } \left[ \frac{i}{2 \pi} f^\ast \Theta( \widehat{h}^\ast) \right]^{\dim V},
$$
where $V_S$ is the locus of points where $f$ is immersive.
\end{prop}
\begin{proof}
We saw in the proof of Theorem \ref{thmsingmetricnonK} that we can locally write $\widehat{h}^\ast \overset{loc}{=} e^{- \Psi}$, with $\Psi$ plurisubharmonic and nowhere equal to $- \infty$ on $Y \setminus p^{-1} (D)$. Consequently, we can write
$$
f^\ast \widehat{h}^\ast \overset{loc}{=} e^{- \Psi \circ f},
$$
with $\Psi \circ f$ plurisubharmonic, and nowhere equal to $-\infty$ outside $f^{-1} \left( p^{-1} (D) \right)$. Since $f (V)$ is not included in $p^{-1} (D)$, this implies that $\Psi \circ f \in \mathrm{Psh} \cap L^1_{loc}$, hence that $f^\ast \widehat{h}^\ast$ induces a singular metric on $f^\ast \mathcal O(1)$, with positive curvature. Therefore, the line bundle $f^\ast \mathcal O(1)$ is pseudo-effective, and we can estimate its volume using Theorem \ref{boucksomthm}. Since $V_S \cup f^{-1} ( p^{-1} (D))$ has zero Lebesgue measure, the absolutely continuous part of $\Theta_{c} ( f^\ast \widehat{h}^\ast)$ is equal to $f^\ast \Theta( \widehat{h}^\ast)$ almost everywhere, which gives the result. 
\end{proof}

\subsection{Bigness of the standard cotangent bundle of a compactification of a ball quotient} \label{sectbigstdcot}

In this section, we prove Theorem \ref{thmsubvar}. We start by recalling some results of \cite{baktsi15}. Let us resume the notations and conventions introduced in Section \ref{sectioncompact}.

\begin{prop} [\cite{baktsi15}] \label{betabound} Let $X'$ be a quotient of $\mathbb B^n$, with $n \geq 2$, and let $X \longrightarrow X'$ be an étale cover, ramifying at order $l$ on the boundary.
Then, for any $\beta > 0$ such that 
\begin{enumerate}
\item $\beta \leq l$ if $n \in \left[| 4, 5 \right|]$ ;
\item $\beta \leq \frac{n+1}{2 \pi} l $ otherwise,
\end{enumerate}
the divisor $K_{\overline{X}} + (1 - \beta) D$ is nef and big.
\end{prop}

Using this proposition, we can immediately apply the base-point free theorem (see \cite{KM98}), to obtain the following lemma. 

\begin{lem} \label{lembasepoint}
With the same hypotheses as in Proposition \ref{betabound}, assume that $\beta$ is a rational number satisfying $\beta < l$ if $n \in \left[|4, 5 \right|]$, and $\beta < \frac{n+1}{2 \pi} l$ otherwise. Then, for any $m \in \mathbb N^\ast$ large enough, the divisor $m \left[ K_{\overline{X}} + (1 - \beta) D \right]$ is base-point free.
\end{lem}

From now on, we will assume that $X$ and $X'$ are as in Theorem \ref{thmsubvar}. Then, $ l > n + 1$ if $l = 4$, and $l > 2 \pi$ in the other cases, so it is possible to find a rational number $\beta$ such that
$$
\beta \in \left] n+1 , \max \left( l, \; \frac{n+1}{2 \pi} l \right) \right[.
$$
In that case, because of Lemma \ref{lembasepoint}, we can write $\beta = \frac{p}{q}$, with $p,q$ large enough so that $L = q (K_{\overline{X}} + D) - p D$ is base-point free.

Consider a subvariety $V$ of $\overline{X}$, not included in $D$. Because of the base-point freeness of $L$,  there exists a section $s \in H^0 \left(\overline{X}, p (K_{\overline{X}} + D ) - q D \right)$, which does not vanish identically on $V$.

Since $\frac{p}{q} > n + 1$, we have $\frac{1}{p} < \frac{1}{(n+1)q}$, so we can choose a real number $\alpha \in \left] \frac{1}{p} , \frac{1}{(n+1) q} \right[$. Let $g$ be the metric induced by $h_{Berg}$ on the line bundle $\mathcal O\left(q \left( K_{\overline{X}} + D\right) \right)$, and let
$$
\phi = \norm{ s }^{2 \alpha}_g.
$$

We can see from \cite[Proposition 1]{mok12}, or from Proposition \ref{smoothframe}, that near the boundary, the metric $g$ is bounded in the local canonical frame $\left( d w_1 \wedge d w_2 \wedge ... \wedge \frac{d w_n}{w_n} \right)^{\otimes q}$ of $\mathcal O( q(K_{\overline{X}} + D))$ as
\begin{equation} \label{boundg}
\norm{\left( d w_1 \wedge ... \wedge \frac{d w_n}{w_n} \right)^{\otimes q}}^2_g \leq C \; \abs{\log \abs {w_n}}^{q(n+1)}.
\end{equation}

Consider the singular metric $\widetilde{h}$ defined on $T_{\overline{X}}$ by $\widetilde{h} = \phi \; h_{Berg}$, and let $h_V$ be its restriction to $T_V$ (at the points where it is defined).

\begin{lem} \label{negcurvature}
On $X \setminus s^{-1} (0)$, $\widetilde{h}$ has negative holomorphic sectional curvature, bounded by a constant $-A$, and negative bisectional curvature. 
\end{lem}
\begin{proof}
Locally on $X \setminus s^{-1} (0)$, we can write
\begin{equation} \label{curvhwidetilde}
\frac{i}{2} \Theta (\widetilde{h}) \; \overset{loc}{=} \;  \frac{i}{2} \D \partial \log \phi \otimes I_n + \frac{i}{2} \Theta(h),
\end{equation}
so, $s | _{X \setminus s^{-1}(0)}$ being a non-vanishing section of the line bundle $\mathcal O(q(K_{\overline{X}} + D))$, we have
\begin{align*}
\frac{i}{2} \D \partial \log \phi & =   \frac{i}{2} \;  \alpha \; \D \partial \log \norm{s}^2_g \\
                                  & = \frac{i}{2} \alpha \; q \; \Theta_{K_{\overline{X} + D}}\\
				 & = - \frac{q  \alpha}{2}\;  \mathrm{Ric}(h_{Berg}) \\
	                         & = q  \alpha  (n+1) \; \omega_{Berg}
\end{align*}

To study the negativity of \eqref{curvhwidetilde}, we can reason locally, in the neighborhood of a point of $X$ corresponding to $0 \in \mathbb B^n$, where $\omega_{Berg}$ admits the expression \eqref{bergmandef}. Then, we can write $\frac{i}{2} \Theta(h_{Berg})$ matricially as
$$
\frac{i}{2} \Theta(h_{Berg})_0 = - \omega_{Berg} \; {I_n} +\frac{i}{2} {}^t \overline{T} \wedge  T,
$$
with $T = \left( dz_1 ... dz_n \right)$. Since $q \alpha (n + 1) < 1$, an easy calculation gives the result.
\end{proof}

Let $V_1 \overset{f_1}{\longrightarrow} V$ be any resolution of $V$. If we let $Z = V_{sing} \cup D \cup s^{-1} (0)$, it is possible to find a resolution $\widetilde{V} \overset{f}{\longrightarrow} V$, dominating $f_1$,  such that the reduced divisor $f^{-1} (Z)_{red}$ has simple normal crossings. Since the sectional and bisectional holomorphic curvatures decrease on submanifolds, we see from Lemma \ref{negcurvature} that $h_V$ has bounded negative sectional curvature and negative bisectional curvature on $\widetilde{V} \setminus f^{-1} (Z)$.

\begin{lem} \label{lembounded} For any $x \in \widetilde{V}$, for any local vector field $\xi$ of $T_{\widetilde{V}}$ defined on a neighborhood of $x$, $\norm{\xi}_{h_V}$ is bounded in a neighborhood of $x$.
\end{lem}
\begin{proof}
If $x \notin f^{-1} (D)$, $h_{Berg}$, considered as a metric on $T_{\overline{X}'}$, is bounded in a neighborhood of $f(x)$, so the result is clear.

If $x \in f^{-1} (D)$, $h_{Berg}$ having Poincaré growth with respect to $D$, we can write for any $p$ near $x$ :
$$
\norm{ f_\ast (\xi) }^2_{Berg} (f(p)) \leq \frac{C} { \abs{w_n}^2 \abs{ \log \abs{w_n}}^2},
$$
where $w_n$ is some local coordinate around $f(x)$, defining $D$. Thus,
\begin{align*}
\norm{\xi}^2_{h_V} & = \phi \;  \norm{f_{\ast} (\xi) }^2\\
		         & \leq C \frac{\norm{s}_g^{2\alpha}}{\abs{w_n}^2 \abs{\log \abs{w_n}}^2 }.
\end{align*}
Since $s$, seen as a section of $\mathcal O( q(K_{\overline{X}} + D))$, vanishes at order $p$ on $D$, this last function is bounded by $\frac{ \abs{w_n}^{2 p \alpha}} { \abs{w_n}^2 \abs{ \log \abs{w_n}}^{2 - 2 (n+1)  q  \alpha}}$, because of \eqref{boundg}. Since $p \alpha > 1$, this gives the result.
\end{proof}

The proof of Theorem \ref{thmsubvar} is now straightforward.

\begin{proof}[Proof of Theorem \ref{thmsubvar}]
Because of Lemma \ref{negcurvature} and Lemma \ref{lembounded}, the metric $h_V$ satisfies all four hypotheses of Theorem \ref{thmsingmetricnonK} on $\widetilde{V}$. Therefore, $\Omega_{\widetilde{V}}$ is big. Since the morphism $\widetilde{V} \longrightarrow V_1$ is proper and birational, it follows that $\Omega_{V_1}$ is big, which ends the proof.
\end{proof}

\begin{rem} There are many other possible choices of singular metrics which could satisfy the hypotheses of Theorem \ref{thmsingmetric}. Let us mention another possible one, in the spirit of \cite{baktsi15}. As explained in Section \ref{localcoordsect} and in \cite{mok12}, each component of $T_b$ of the boundary admits a tubular neighborhood $\Omega_b^{(N)}$, for $N$ large enough, on which $\omega_{Berg}$ is given by the potential $l(w)= - \frac{4 \pi}{\tau_b} \log \norm{w}^2$, i.e. $\omega_{Berg} = \frac{i}{2} \overline{\partial} \partial \log l$.

We define a metric on $T_X$ by $\widetilde{h} = e^{- \chi(l)} h_{Berg}$ on $\Omega_{b}^{(N_b)}$, where $\chi : \mathbb R \longrightarrow \mathbb R$ is a smooth function such that $t \mapsto \chi(t) + \log t$ approximates $t \mapsto \log(t)$ on $]0, N_b]$ and the tangent line to $t \mapsto  \log(t)$ at $N_b$ on $]N_b, + \infty[$.

Now, $\widetilde{h}$ equals $h$ outside $\Omega_b^{(N)}$, and since $t \mapsto  - \left( \log(t) + \chi(t) \right)$ is convex, we see that
$$
\omega_{Berg} + \frac{i}{2} \partial \overline{\partial} \chi(l) = - \frac{i}{2} \partial \overline{\partial} ( \log l + \chi(l) ) \geq 0.
$$
Thus, the bisectional curvature of $h_{Berg}$ being larger or equal to $-4$, we conclude, e.g. by applying \eqref{curvatureprojectivized}, that the holomorphic sectional curvature of the metric $\widehat{h}$, induced by $\widetilde{h}$ on $\mathcal O(1) \longrightarrow \mathbb P(T_X)$, is non-negative. Now, if $p = (x, \left[v\right])$ is a point of $\mathbb P(T_{\overline{X}})$ with $ x \in D$, we have the following asymptotic bound at $p$:
$$
\log \widehat{h}  \leq - \log \abs{w_n}^2 - \chi(l) + \underset{\abs{w_n} \longrightarrow 0}{O(1)}  
	 \leq \abs{\, \log \abs{w_n}^2 \, } - \frac{l}{N} + \underset{\abs{w_n} \longrightarrow 0}{O(1)},
$$
where we used the fact that the eigenvalues of $h_{Berg}$ have growth at most $- \log|w_n|^2 -  \log (-  \log |w_n|^2)$ near the boundary, by \eqref{diagonalformeq}. Finally, $l(w) \underset{\abs{w_n} \sim 0}{\sim} \frac{\tau}{4 \pi} (-\log \abs{w_n}^2)$, and we see that $\widetilde{h}$ will be bounded provided $\frac{\tau}{4 \pi N} < 1$.

If we can take uniformly $N_b < \frac{\tau_b}{4 \pi}$ for each cusp $C_b$, the singular metric $\widetilde{h}$ will be bounded everywhere, and $\widehat{h}$ will satisfy the hypotheses of Theorem \ref{thmsingmetricnonK}. By \cite{parker1998}, we can take in any case $N_b = \frac{\tau_b}{2}$ uniformly. Now, consider an \'etale cover $X \longrightarrow X'$ ramifying at an order $l \geq 7$ over a boundary component $T_{b'} \subset \overline{X'}$. Let $T_b \subset \overline{X}$ be a boundary component projecting to $T_{b'}$. We see from the description in local coordinates that $\tau_{b'} = l\, \tau_{b}$, and that $N_{b} = N_{b'}$ is an admissible horoball size at the cusp $b$. Consequently, we have $N_{b} = \frac{1}{l} \frac{\tau_{b'}}{2} < \frac{\tau_b}{4 \pi}$, and the singular metric $\widetilde{h}$ on $T_{\overline{X}}$ satisfies all our requirements. 

While we could have used this choice of singular metric to prove Theorem \ref{thmsubvar}, our previous choice uses the bigness of $K_{\overline{X}}$ when $n \geq 4$, provided by \cite{baktsi15}. This gives a better bound in dimension 4; we would similarly obtain the better bound $l \geq 5$ in dimension 3 if it were proved that all toroidal compactifications of this dimension are of general type.  
\end{rem} 

\section{Birational transformation between logarithmic and standard projectivized tangent bundles} \label{biratsection}

 In this section, we introduce a construction that will be useful in Section \ref{nefnesssection}, when we study the nefness of the cotangent bundle of a toroidal compactification. \\

The plan of our work in the next sections is straightforward: we will first show that the logarithmic cotangent bundle of a toroidal compactification is nef, using Proposition \ref{integralsubmanifold}, and then use this result to study the standard cotangent bundle. To do this, we will resolve the birational map $\mathbb P\left( T_{\overline{X}} (- \log D) \right) \dashrightarrow \mathbb P\left(T_{\overline{X}} \right)$ into a sequence of two blowing-ups:
\begin{equation} \label{eqresolution}
\mathbb P\left( T_{\overline{X}} (- \log D) \right) \; \overset{\pi}{\longleftarrow} \; \widetilde{Y} \; \overset{\pi_0}{\longrightarrow} \; \mathbb P\left( T_{\overline{X}} \right).
\end{equation}

With this construction, it will not be hard to express the pullbacks of the two tautological line bundles onto $\widetilde{Y}$, in term of each other. Therefore, we will be able later on to deduce a condition for $\Omega_{\overline{X}}$ to be nef, knowing that $\Omega_{\overline{X}} (\log D)$ is nef.

In the rest of the section, we describe the resolution \eqref{eqresolution}: in fact, it holds for more general log-pairs than toroidal compactifications. Actually, for any log-pair $(X , D)$ with \emph{smooth boundary}, there is a canonical way to resolve the map $\mathbb P( T_{X} ( -\log D)) \dashrightarrow \mathbb P(T_{X})$, by blowing up a single smooth analytic subset in each of these two manifolds. \\

For the rest of the section, $(X, D)$ will be a log-pair with smooth boundary. We will denote by $Y = \mathbb P( T_X(-\log D))$ the projectivized bundle of the logarithmic tangent bundle, with its associated tautological bundle $\mathcal O_Y (1)$. In the same way, let $Y_0 = \mathbb P(T_X)$, and let $\mathcal O_{Y_0}(1)$ be its tautological bundle. We will denote by $p : Y \longrightarrow X$ and $p_0 : Y_0 \longrightarrow X$ the canonical projections. \\

On $(X, D)$, we have the usual logarithmic cotangent exact sequence:
\begin{equation} \label{logcotexact}
0 \longrightarrow \Omega_{X} \longrightarrow \Omega_X(\log D) \overset{\mathrm{res}}{\longrightarrow} \mathcal O_D \longrightarrow 0, 
\end{equation}
the last arrow being the Poincaré residue map. The surjective morphism $\Omega_{X} (\log D) \longrightarrow \mathcal O_{D}$ induces a section of the projection $p^{-1} (D) = \mathbb P^\ast \left( \Omega_{X} (\log D) |_D \right) \longrightarrow D$, whose image we will denote by $Z$.\\

In a similar way, we can write the following exact sequence:
\begin{equation} \label{exactseq2} 
0 \longrightarrow \Omega_X(\log D) \otimes \mathcal O(-D) \longrightarrow \Omega_X \longrightarrow \Omega_D \longrightarrow 0,
\end{equation}
where the last arrow is induced by the restriction to $T_D$, and the first arrow is given in local coordinates by
$$
\left( \sum_i {v_i dz_i} + v_n \frac{d z_n}{z_n} \right) \otimes z_n \mapsto \sum_i {(z_n v_i) dz_i} + v_n dz_n,
$$
where $(z_1, ..., z_n)$ are local coordinates such that $z_n$ is an  equation for $D$. Exactly as before, the last arrow induces a closed immersion $\mathbb P_D(T_D) \cong \Ps{X}(\Omega_D) \hookrightarrow \mathbb P_X(T_X)$, whose image we will denote by $Z_0$.

With these notations, the following result can be proved in a straightforward way.

\begin{prop} \label{resolution} The natural birational map $Y \dashrightarrow Y_0$ induces an isomorphism of projective manifolds:
$$
\mathrm{Bl}_Z Y \overset{\simeq}{\longrightarrow} \mathrm{Bl}_{Z_0} Y_0.
$$
Moreover, if $\pi : \mathrm{Bl}_Z Y \longrightarrow Y$ and $\pi_0 : \mathrm{Bl}_{Z_0} Y_0 \longrightarrow Y_0$ denote the respective blowing-ups, then the strict transform of $p^{-1} (D)$ corresponds under this isomorphism to the exceptional divisor of $\pi_0$. In the same manner, the strict transform of $p_0^{-1} (D)$ under $\pi_0$ corresponds to the exceptional divisor of $\pi$.
\end{prop}

Now, let $\widetilde{Y} = \mathrm{Bl}_{Z} Y$, which is canonically identified to $\mathrm{Bl}_{Z_0} Y_0$. With the same notations as before, let $E, E_0 \subset \widetilde{Y}$ be the exceptional divisors of the respective projections $\pi, \pi_0$. Keeping track of the tautological line bundles of the two blowing-ups, we can prove the following:

\begin{prop} \label{propexceptional} On $\widetilde{Y}$, we have the following isomorphism of line bundles:
$$
\pi^\ast \mathcal O_Y(1) \simeq \pi_0^\ast \mathcal O_{Y_0}(1) \otimes_{\mathcal O_{\widetilde{Y}}} \mathcal O(E)
$$
\end{prop}

It is easy to see that $\left[ w_n \frac{\partial}{\partial w_n} \right]$ realizes a global non-vanishing section of $\mathcal O_Y(1)$ on $Z$. Pulling back via $\pi$, we find:
\begin{prop} \label{restricttrivial} The restriction of $\pi^\ast \mathcal O_Y(1)$ to $E$ is trivial.
\end{prop}

We see from this result that if $W \subset Y_0$ is a subvariety with strict transform under $\pi_0$ denoted by $\widetilde{W}$, we can compute the maximal intersection of $\mathcal O_{Y_0} (1)$ with $W$ in terms of  intersection numbers of $\widetilde{W}$ with $\pi^\ast \mathcal O_{Y}(1)$ and $E$. Indeed, we have
\begin{align} \nonumber
c_1 (\mathcal O_{Y_0})^{\dim W} \cdot W & = c_1 (\pi_0^\ast \mathcal O_{Y_0})^{\dim W} \cdot \widetilde{W} \\ \nonumber
					& = c_1 (\pi^\ast \mathcal O_{Y} \otimes \mathcal O(-E))^{\dim W} \cdot \widetilde{W} \\ \nonumber
					& = c_1 (\pi^\ast \mathcal O_{Y})^{\dim W} \cdot \widetilde{W} + (-1)^{\dim W} E^{\dim W} \cdot \widetilde{W} \\
					& = c_1 (\mathcal O_{Y})^{\dim W} \cdot \pi( \widetilde{W} ) + (-1)^{\dim W} E^{\dim W} \cdot \widetilde{W}. \label{equalityintersectionnumbers}
\end{align}
We will see in the next sections that in the case where $(\overline{X}, D)$ is a toroidal compactification, we can estimate the first term of the right hand side of this last equation, in terms of the Bergman metric on $\overline{X} \setminus D$. As for the second member, we can prove a more general result, for any log-pair with smooth boundary. To estimate the intersection numbers with $E \subset \widetilde{Y}$, we can use the following result, which determines the normal bundle to $Z$. 

\begin{prop} \label{isomnormal} There is a canonical isomorphism
\begin{equation} \label{isomnormaleq}
N^\ast_{Z / Y} \simeq p^\ast \left( \left. \Omega_{X} \right|_D \right).
\end{equation}
\end{prop}

The exceptional divisor $E$ is isomorphic, as a $D$-scheme, to $\mathbb P(N_{Z / Y}) = \mathbb P^\ast( N^\ast_{Z / Y})$. We saw in Proposition \ref{resolution} that the canonical isomorphism $\mathrm{Bl}_Z Y \cong \mathrm{Bl}_{Z_0} Y_0$ sends $E$ to the strict transform of $p_0^{-1}(D)$ under $\pi_0$. Since $\mathbb P(T_D)$ has codimension one in $p_0^{-1} (D)$, this strict transform is isomorphic to $p_0^{-1}(D)$. Actually, we have the following proposition:

\begin{prop} \label{structureofE}
The projection $\pi_0$ induces an isomorphism
\begin{equation} \label{isomexceptional}
\left. \pi_0 \right|_{E}: E \longrightarrow p_0^{-1} (D) \simeq \mathbb P(\left. \Omega_X \right|_D),
\end{equation}
determined by the isomorphism of $\mathcal O_Z$-modules \eqref{isomnormaleq}.
\end{prop}

\section{Nefness of the cotangent bundles} \label{nefnesssection}

In the rest of the text, $(\overline{X}, D)$ will be a toroidal compactification of a ball quotient. \\

With what has been introduced until now, we can  use the results of \cite{baktsi15} to determine a condition for $\Omega_{\overline{X}}$ to be nef. For this, we let $\overline{Y} = \mathbb P_{\overline{X}}(T_{\overline{X}}(-\log D))$, with its canonical projection $p$ and tautological bundle $\mathcal O(1)_{\log}$, and $\overline{Y}_0 = \mathbb P_{\overline{X}}(T_{\overline{X}})$, with its projection $p_0$ and tautological bundle $\mathcal O(1)_0$. We start by proving that $\Omega_{\overline{X}} (\log D)$ is always nef.

\begin{proof} [Proof of Theorem \ref{nefcotangentlog}] Let $C \subset \overline{Y}$ be an irreducible curve. If $C \not\subset p^{-1} (D)$, it follows from Proposition \ref{integralsubmanifold} that $c_1 \mathcal O(1)_{\log} \cdot C \geq 0$.

If $C \subset p^{-1} (D) = \mathbb P( T_{\overline{X}} (- \log D) |_{D} )$, the result is given by the next lemma.
\end{proof}

\begin{lem} The restriction $\left. \Omega_{\overline{X}} (\log D) \right|_D$ is nef.
\end{lem}
\begin{proof} This is a basic application of the properties of the logarithmic conormal sequence. Recall that since $D$ is smooth, we have the following exact sequence of locally free $\mathcal O_D$-modules:
$$
0 \longrightarrow \Omega_D \longrightarrow \Omega_{\overline{X}}(\log D) \otimes_{ \mathcal O_{\overline{X}}} \mathcal O_D \longrightarrow \mathcal O_D \longrightarrow 0.
$$

This can be seen directly in coordinates, the second map sending $\sum_{1 \leq i \leq n} a_i dz_i + a_n \frac{dz_n}{z_n}$ to $a_n$, or by tensoring the Poincar\'{e} residue exact sequence by $\mathcal O_D$.

Since the boundary is made of abelian varieties, $\Omega_D$ is trivial on any component of $D$. Consequently, the vector bundle $\Omega_{\overline{X}}(\log D) |_{D}$ is an extension of trivial bundles, hence is nef.
\end{proof}

Let us mention the following result, first step in our study of the nefness of $\Omega_{\overline{X}}$.

\begin{prop} \label{nefbound}  When restricted to $D$, the cotangent bundle $\Omega^1_{\overline{X}}$ is nef.
\end{prop}

\begin{proof} As stated in \cite{mok12}, for any component $D_b$ of $D$, the neighborhoods $\Omega^{(N)}_b$ introduced in Section \ref{sectioncompact} are isomorphic to tubular neighborhoods of the zero section of the normal bundle $N_b \longrightarrow D_b$. Consequently, we have
\begin{equation} \label{exactsequenceboundary}
\left. \Omega_{\overline{X}} \right|_{D_b} \simeq N^\ast_b \oplus \Omega_{D_b} \simeq N^\ast_b \oplus \mathcal O_D^{\oplus n-1},
\end{equation}
since $D_b$ is an abelian variety. Moreover, for any such component $D_b$, the conormal bundle $N^\ast_b$ is positive (\cite{mok12}). Thus, $\left. \Omega_{\overline{X}} \right|_{D_b}$ is sum of a trivial bundle and of an ample bundle on $D_b$, hence is nef.
\end{proof}

We will now make use of the results we proved in Section \ref{biratsection} to estimate the intersection numbers of the type $c_1 \mathcal O(1)_0 \cdot C$, where $C$ is a curve of $\overline{Y}_0$, not included in the boundary. To do this, we will pull back all our objects to the blowing-up $\mathrm{Bl}_Z \overline{Y}$.  Let $\widetilde{Y}$ denotes this blowing-up, that we endow with its natural projections $\pi$ and $\pi_0$, respectively onto $\overline{Y}$ and $\overline{Y}_0$.

\begin{prop} \label{ineqprop} Let $C \subset \overline{Y}$ be a curve such that $p_0(C) \not\subset D$. Then
$$
c_1 \mathcal O(1)_0 \cdot C \geq \left( \frac{1}{n+1} \left( K_{\overline{X}} + D \right) - D \right) \cdot p_0(C).
$$
\end{prop}
\begin{proof}
We denote by $\widetilde{C}$ the proper transform of the curve $C$ by the blowing-up $\pi_0$. Then Proposition \ref{propexceptional} gives
\begin{align*}
c_1 \mathcal O(1)_0 \cdot C & = \pi^{\ast}_0 \left( c_1 \mathcal O(1)_0 \right) \cdot \widetilde{C} \\ 
		   & = \pi^{\ast} \left( c_1 \mathcal O(1)_{\log} - E \right) \cdot \widetilde{C},
\end{align*}

Moreover, thanks to Proposition \ref{integralsubmanifold}, we obtain
\begin{align*}
\pi^{\ast} \left( c_1 \mathcal O(1)_{\log} \right) \cdot \widetilde{C} & = c_1 \mathcal O(1)_{\log} \cdot \pi(\widetilde{C}) \\
						          & \geq \int_{ Y \cap \pi (\widetilde{C})} \frac{i}{2 \pi} \Theta(\widehat h^\ast)  \\
						          & = \int_{Y \cap C} \frac{i}{2 \pi} \Theta(\widehat h^\ast).
\end{align*}
The Bergman metric being of constant sectional curvature equal to $-4$ with our choice of normalization, the following equality is true at any point $(x, \left[ v \right]) \in Y$, for any $\xi \in T_{(x, \left[v\right])} Y$:
$$
\frac{i}{2 \pi} \Theta(\widehat h^{\ast})_{(x, [v])} \cdot(\xi, \xi)  \geq -\frac{i}{2 \pi} \frac{\Theta_{T_X}(v, v, \xi, \xi)}{\norm{v}^2}
						           \geq \frac{1}{ \pi} \omega_{Berg}( \xi, \xi),
$$
so
$$
\pi^\ast \left( c_1 \mathcal O(1)_{\log} \right) \cdot \widetilde{C} \geq \int_{X \cap p_0 (C) } \frac{1}{\pi} \omega_{Berg}.
$$
However, because of Proposition \ref{propintbergman}, we obtain
$$
\int_{X \cap p_0 (C)} \frac{1}{\pi} \omega_{Berg}  =  \frac{1}{n+1} (K_{\overline{X}} + D)\cdot C.
$$
 
Besides, since $E$ is an irreducible component of $(p \circ \pi)^{-1}(D)$, we have 
$$
E \cdot \widetilde{C} \geq D \cdot (p \circ \pi)(\widetilde{C}) = D \cdot p_{0}(C).
$$
\end{proof}
We can now prove our main result on the nefness of $\Omega_{\overline{X}}$.
\begin{proof} [Proof of Theorem \ref{thmnef} ]
Consider an irreducible curve $C \subset \overline{Y_0}$. First assume that $p_0(C) \subset D$. According to Proposition \ref{nefbound}, the bundle $\left. \Omega_{\overline X} \right|_{D}$ is nef. Since $C$ can be seen as a curve of the projective space $\mathbb P(\left. T_{\overline X} \right|_{D})$, we see that $\int_{C} {c_1 \mathcal O(1)_0} \geq 0$.

Assume now that $C \cap Y \neq \emptyset$. Then, according to Proposition \ref{ineqprop}, we have
\begin{equation*}
\int_C {c_1 \mathcal O(1)_{0}} \geq \frac{1}{n+1} \int_{p_0(C)} {c_1\left(K_{\overline{X}} + \left( 1 - (n + 1) \right) D \right)}.
\end{equation*}
In addition, since $\sigma$ ramifies to an order larger than $7$ along the boundary, we have
$$
\int_{p_0(C)} { c_1 \left( K_{\overline{X}} + \left( 1 - (n + 1) \right) D \right) } \geq \int_{p_0(C)} {c_1 \left( \sigma^\ast \left( K_{\overline X'} + D' \right) - \frac{n+1}{7} \sigma^{\ast} D' \right)}.
$$
Therefore,
\begin{equation} \label{ineqcurves}
\begin{split}
\int_{p_0(C)} {c_1\left(K_{\overline{X}} + \left( 1 - (n + 1) \right) D \right)}  & \geq \left( \deg \sigma \right) \int_{\sigma (p_0(C))} {c_1 \left(K_{\overline X'} + \left(1 - \frac{n+1}{7} \right) D' \right)} \\
& \geq \left( \deg \sigma \right) \int_{\sigma (p_0(C))} {c_1 \left(K_{\overline X'} + \left(1 - \frac{n+1}{2 \pi} \right) D' \right)},
\end{split}
\end{equation}
since $\left( D \cdot \left( \sigma \circ p_0 (C) \right)  \right) \geq 0$ (the divisor $D$ and the curve $\sigma \circ p_0(C)$ are in normal intersection). The line bundle $K_{\overline X} + \left(1 - \frac{n+1}{2 \pi} \right) D$ is nef by \cite{baktsi15}, so the last term of \eqref{ineqcurves} is non-negative, which gives the result.
\end{proof}

\section{Immersed submanifolds of $\overline{X}$}

We know turn to the proof of Corollary \ref{corolamplemod}. As recalled in \cite{dicerbo2015}, there is a simple criterion to prove that a line bundle is ample modulo an analytic subset.

\begin{prop} [ cf. \cite{dicerbo2015}] \label{ampleboundary} Let $(X, D)$ be a logarithmic pair, and let $L$ be a nef line bundle on $X$. If for any subvariety $V$ of $X$, not included in $D$, we have $c_1(L)^{\dim V} \cdot V > 0$, then $L$ is ample modulo $D$.
\end{prop} 

Now, consider an \'etale cover of a ball quotient $X \longrightarrow X'$, ramifying at order larger than $7$ on the boundary. Let $\overline{V} \subset \overline{X}$ be an immersed subvariety, not included in $D$. Let $q : \mathbb P(T_{\overline{V}}) \longrightarrow V$ be the natural projection. There is a well defined immersion $\mathbb P(T_{\overline{V}}) \overset{f}{\longrightarrow} \mathbb P(T_{\overline{X}})$, and $f^\ast \mathcal O_{\mathbb P(T_{\overline{X}})}(1) = \mathcal O_{\mathbb  P(T_{\overline{V}})}(1)$ is nef because of Theorem \ref{thmnef}.

It follows from the discussion of Section \ref{sectbigstdcot} that $\mathcal O_{\mathbb P(T_{\overline{X}})}(1)$ admits a singular metric $\widehat{h}$ with positive curvature, and such that $- \log \widehat{h}$ is bounded from above near the boundary. Pulling back to $\mathbb P(T_V)$, we see that the same holds for $f^\ast \widehat{h}$. In particular, the metric $f^\ast \widehat{h}$ has positive curvature in the sense of currents, and the absolutely continuous part of this current is given by the curvature on the open part of $\overline{V}$.

Let $\overline{W} \subset \mathbb P(T_{\overline{V}})$ be a subvariety which is not included in $q^{-1}(D)$, and call $W = \overline{W} \cap f^{-1} \left( \mathbb P(T_X) \right)$ its open part. We can apply Theorem \ref{boucksomthm} to the metric $f^\ast \widetilde{h}$, to get
$$
c_1 \mathcal O_{\mathbb P(T_{\overline{V}})}(1)^{\dim W} \cdot \overline{W} = \mathrm{vol}(\mathcal O_{\mathbb P(T_V)}(1) |_W ) \geq \int_{W} \left( \frac{i}{2 \pi} \Theta(f^\ast \widetilde{h})^{\dim W} \right) > 0,
$$
where we used the fact that $\mathcal O_{\mathbb P(T_{\overline{V}})}$ is nef to obtain the first equality. Thus, applying Proposition \ref{ampleboundary}, we immediately obtain Corollary \ref{corolamplemod}.

\subsection{Volume and numerical intersection numbers}

In this last section, we would like to show how we can obtain lower bounds on the volume of $\Omega_{\overline{X}}$, under the hypotheses of Theorem \ref{thmnef}. Let $X \longrightarrow X'$ an étale cover of a ball quotient, such that $\overline{X} \longrightarrow \overline{X'}$ ramifies at order $7$. If $\overline{V}$ is a smooth compact manifold of dimension $p \leq n$, and if we have a (non-necessarily injective) immersion $f : \overline{V} \longrightarrow \overline{X}$, with $f(\overline{V}) \nsubset D$, then there is an induced holomorphic map $\mathbb P(T_{\overline{V}}) \overset{\widetilde{f}}{\longrightarrow} \mathbb P(T_{\overline{X}})$. By Theorem \ref{thmnef}, the line bundle $\mathcal O_{\mathbb P(T_{\overline{V}})}(1) = \widetilde{f}^\ast \mathcal O_{\mathbb P(T_{\overline{X}})}(1)$ is nef, so 
$$
\mathrm{vol} (\Omega_{\overline{V}}) = \mathrm{vol} (\mathcal O_{\mathbb P(T_{\overline{V}})}(1)) = c_1 \left( \mathcal O_{\mathbb P(T_{\overline{V}})}(1) \right)^{2p-1}.
$$

We will briefly explain how we can compute lower bounds to these numbers. Since the essential technical part of the computations is very close to \cite{div16}, so we will only present the main ideas leading to them.
 
Let $\overline{W} = \mathbb P(T_{\overline{V}})$, and let $q_{\overline{W}} : \overline{W} \longrightarrow \overline{V}$ be the natural projection. Denote by $V = f^{-1}(X)$ the open part of $\overline{V}$. We resume our previous notations: let $\overline{Y}_0 = \mathbb P(T_{\overline{X}} (- \log D))$, and $\overline{Y} = \mathbb P(T_{\overline{X}})$, with their respective line bundles $\mathcal O(1)_{\log}$ and $\mathcal O(1)_0$. Let $\widetilde{Y}= \mathrm{Bl}_Z \overline{Y}$, where $Z$ is the subvariety of $\overline{Y}$ introduced in Section \ref{biratsection}, and let $E \subset \widetilde{Y}$ be the exceptional divisor. We have the following fibre square, where $\widetilde{W}$ is the blowing-up of $\overline{W}$ along $\widetilde{f}^{-1}(Z)$:

\begin{diagram}
\widetilde{W}         &\rTo^{\widetilde{g}}   &\widetilde{Y}\\
\dTo_{q}  &           &\dTo_{\pi_{0}}\\
\overline{W}        &\rTo^{\widetilde{f}}   &\overline{Y}
\end{diagram}

By \eqref{equalityintersectionnumbers}, we obtain the following inequalities of intersection numbers:
\begin{equation} \label{intersectiontangent} 
 \left[ c_1 \mathcal O(1)_{\overline{W}} \right]^{2p -1} = \left[ \widetilde{g}^\ast \; \pi^\ast c_1 \mathcal O(1)_{\log} \right]^{2p-1} + (-E)^{2p - 1} \cdot \widetilde{W}.
\end{equation}
By Theorem \ref{nefcotangentlog}, the line bundle $\widetilde{g}^\ast \mathcal O(1)_{\log}$ is nef on $\widetilde{Y}$, and
\begin{align*}
\left[ \widetilde{g}^\ast \; \pi^\ast c_1 \mathcal O(1)_{\log} \right]^{2p - 1} & = \mathrm{vol} \left( \widetilde{g}^\ast \; \pi^\ast \mathcal O(1)_{\log} \right) \\
	& \geq  \int_{W} \left[ \widetilde{g^\ast} \left( \frac{i}{2 \pi} \Theta(\widehat{h^\ast}) \right)^{2 p - 1} \right] \tag*{(by Proposition \ref{integralsubmanifold})}.
\end{align*}
Here, $W \cong \widetilde{f}^{-1} (Y)$ denotes the open part of $\overline{W}$. The last quantity, which is easily seen to be equal to $\int_{\widetilde{f}^{-1} (Y)}  \widetilde{f}^\ast \left[  \frac{i}{2 \pi} \Theta(\widehat{h^\ast})\right]^{2p-1}$, can be bounded from below by a direct computation.  

\begin{fact}
The following inequality holds:
\begin{equation} \label{interlognumber}
\int_{\widetilde{f}^{-1}(Y)} \widetilde{f}^\ast \left[ \frac{i}{2 \pi} \Theta(\widehat{h}^\ast) \right]^{2p - 1} \geq  \deg(f) \binom{2p - 1}{p}  \frac{\left( K_{\overline{X}} + D\right)^p \cdot f(W)}{(n+1)^p}.
\end{equation}
\end{fact}
\begin{proof}
This is an explicit computation, following the ideas of \cite{div16}. Using the expression of the Bergman metric \eqref{bergmandef}, we can integrate \eqref{interlognumber} on the fibers of the projection $q_{\overline{W}}$. Since the holomorphic sectional curvature decreases on subvarieties, everything is finally bounded from below, up to a normalization constant, by the volume of $V$ with respect to $h_{Berg}$. Finally, since $\mathrm{Ric}(h_{Berg}) = 2(n+1) \omega_{Berg}$, this volume can be related to $c_1 (K_{\overline{X}} + D)^p \cdot f(V)$. The main technical part is to keep track of the proportionality constants.
\end{proof}

\begin{rem} \begin{enumerate}
\item When $\dim V = n$, the same computations permit to obtain the more precise inequality: 
$$
c_1 \mathcal O(1)_{\log}^{2n-1} \geq \binom{2n}{n} \frac{ \left( K_{\overline{X}} + D \right)^n }{(n+1)^n},
$$
which is in fact an equality by Hirzebruch's proportionality principle in the non-compact case (see \cite{mum77}).
\item If $\dim V = 1$, we can refine the above computations to obtain the following inequality:
$$
\int_{\widetilde{f}^{-1}(Y)} \widetilde{f}^\ast \left[ \frac{i}{2 \pi} \Theta(\widehat{h}^\ast) \right] \geq 2 \deg(f) \frac{ \left( K_{\overline{X}} + D  \right) \cdot f(\overline{V}) } {n +1}. \\
$$
\end{enumerate}
\end{rem}

Now, we can compute the term $(-E)^{2p -1} \cdot \widetilde{W}$, appearing in \eqref{intersectiontangent}. Let $E_W = \widetilde{g}^{-1} (E)$. Then
$$
(-E)^{2 p - 1} \cdot \widetilde{W} = - \left( \left. E \right|_E \right)^{2(p-1)} \cdot \left(\left. \widetilde{W} \right|_E\right) = -  \left( g^\ast \left. E \right|_{E} \right)^{2(p-1)} \cdot E_W.
$$
Using Proposition \ref{structureofE}, we see easily that there are isomorphisms $E \simeq \mathbb P\left(\left. T_{\overline{X}}\right|_D\right)$, and $E_W \simeq \mathbb P(\left. T_{\overline{V}} \right|_{f^{-1}(D)})$, the morphism $E_W \longrightarrow E$ being induced by the inclusion $T_{\overline{V}} |_{f^{-1} (D)} \hookrightarrow T_X |_D$. Finally, the functoriality of tautological bundles under pull-backs gives
\begin{equation} \label{formulaintersectionW0}
\left( g^\ast \left. E \right|_{E} \right)^{2(p-1)} \cdot E_W = \int_{E_W} c_1 \left( \widetilde{f}^\ast_D \mathcal O_E (1) \right)^{2(p -1)} = \int_{E_W} c_1 \left( \mathcal O_{E_W} (1) \right)^{2(p-1)}.
\end{equation}
We can bound this last number from above, for example by constructing a natural metric on the tautological bundle of $E \cong \mathbb P(T_{\overline{X}} |_D)$. To do this, recall that each component of $D$ admits a tubular neighborhood in $\overline{X}$, which gives the isomorphism \eqref{exactsequenceboundary}.  As before, we can integrate the curvature of such a metric on the fibers of the projection $E_W \longrightarrow f^{-1} (D)$, to estimate the quantity \eqref{formulaintersectionW0} in terms of intersection numbers on $f^{-1}(D)$. Then
\begin{equation} \label{ineqintersect}
(-E)^{2p-1} \cdot \widetilde{W} \geq - \int_{f^{-1} (D)} {\left[ - \frac{i}{2\pi} \Theta(N_{D / \overline{X}}) \right]^{p-1}} = - \left(\left. -D \right|_{\overline{W}}\right)^{p-1} = (-D)^p \cdot \overline{V}.
\end{equation}

\begin{rem} In the case  $\overline{V} = \overline{X}$, we can compute exactly the top self-intersection $(-E)^{2n -1}$, which is given by the top Segre class of $T_{\overline{X}} |_{D}$. By \eqref{exactsequenceboundary}, this implies that $(-E)^{2n -1} = (-D)^n$.
\end{rem}

Putting everything together,  we have the following result:

\begin{prop} \label{finalineq} Let $\overline{X}$ be a toroidal compactification, and let $\overline{V} \overset{f}{\longrightarrow} \overline{X}$ be an immersion of a smooth manifold $\overline{V}$ of dimension $p$, not necessarily injective, such that $f(\overline{V}) \not\subset D$. Then if $\mathcal O_{\overline{V}}(1)$ is the tautological bundle of $\mathbb P(T_{\overline{V}})$, we have the following inequality:
\begin{equation} \label{ineqimmersed}
c_1 \mathcal O_{\overline{V}} (1)^{2p - 1} \geq \left[ \binom{2p - 1}{p} \frac{1}{(n+1)^p} (K_{\overline{X}} + D)^p + (-D)^p \right] \cdot f_\ast \left[\overline{V}\right], 
\end{equation}
where $f_\ast \left[ \overline{V} \right] = \deg(f) \left[f(\overline{V}) \right]$ denotes the image cycle of $\overline{V}$. If $p = 1$, we have the more precise inequality
\begin{equation} \label{ineqcurves2}
\deg K_{\overline{V}} = c_1 \mathcal O_{\overline{V}} (1) \geq \left[ \frac{2}{n+1} (K_{\overline{X}} + D)  - D \right] \cdot f_\ast \left[\overline{V}\right], 
\end{equation}
and if $p = n$, we have the equality (via \cite{mum77}).
\begin{equation} \label{ineqwhole}
c_1 \mathcal O_{\overline{V}} (1)^{2n - 1} =  \left[ \binom{2n}{n} \frac{1}{(n+1)^n} (K_{\overline{X}} + D)^n + (-D)^n \right] \cdot f_\ast \left[\overline{V}\right]. \\
\end{equation}
In particular, if $\overline{X}$ satisfies the hypotheses of Theorem \ref{thmnef}, these inequalities give lower bounds on $\mathrm{vol} \left( \Omega_{\overline{V}} \right)$.
\end{prop} 

\bibliographystyle{amsalpha}
\bibliography{biblio} 
\textsc{Benoît~Cadorel, Aix Marseille Université, CNRS, Centrale Marseille, I2M, UMR~7373, 13453 Marseille, France} \par\nopagebreak
  \textit{E-mail address}: \texttt{benoit.cadorel@univ-amu.fr}
\vfill
\end{document}